\documentclass[12pt,a4paper,leqno]{article}

\usepackage[nottoc]{tocbibind}
\usepackage{latexsym}
\usepackage{amssymb,exscale}
\usepackage[centertags]{amsmath}
\usepackage{amsthm}
\usepackage{paralist}
\usepackage{footnpag}
\usepackage[all]{xy}

\usepackage[linktocpage]{hyperref}

\numberwithin{equation}{section}

\swapnumbers
\theoremstyle{definition}
\newtheorem{theorem}[equation]{Theorem}
\newtheorem{lemma}[equation]{Lemma}
\newtheorem{corollary}[equation]{Corollary}

\newtheorem{definition}[equation]{Definition}
\newtheorem{proposition}[equation]{Proposition}

\renewcommand{\phi}{\varphi}

\newcommand{\I}{{\rm i}}
\newcommand{\E}{{\rm e}}

\newcommand{\ti}{\tilde}

\renewcommand{\(}{\bigl(}
\renewcommand{\)}{\bigr)\vphantom{)}}

\newcommand{\ip}[2]{\langle#1,#2\rangle}

\newcommand{\One}{{1\hskip-2.5pt{\rm l}}}

\newcommand{\ga}{\gamma}

\newcommand{\al}{\alpha}
\newcommand{\be}{\beta}

\newcommand{\Ec}{\mathcal E}

\newcommand{\la}{\lambda}

\newcommand{\Q}{\mathbb Q}
\newcommand{\R}{\mathbb R}

\newcommand{\C}{\mathbb C}
\newcommand{\Z}{\mathbb Z}

\newcommand{\dimensional}[1]{$#1$\nobreakdash-\hspace{0pt}dimensional}

\makeatletter
\renewcommand*\l@section[2]{%
  \ifnum \c@tocdepth >\z@
    \addpenalty\@secpenalty
    \addvspace{0.25em \@plus\p@}%
    \setlength\@tempdima{2.5em}%
    \begingroup
      \parindent \z@ \rightskip \@pnumwidth
      \parfillskip -\@pnumwidth
      \leavevmode \bfseries
      \advance\leftskip\@tempdima
      \hskip -\leftskip
      #1\nobreak\hfil \nobreak\hb@xt@\@pnumwidth{\hss #2}\par
    \endgroup
  \fi}
\renewcommand*\numberline[1]{\hb@xt@\@tempdima{\hfil#1\hskip1em}}
\makeatother

\begin{document}

\title{Subproduct systems of Hilbert spaces: dimension two}

\author{Boris Tsirelson}

\date{}
\maketitle

\begin{abstract}
A subproduct system of two-dimensional Hilbert spaces can generate an Arveson
system of type $ I_1 $ only. All possible cases are classified up to
isomorphism. This work is triggered by a question of Bhat: can a subproduct
system of \dimensional{n} Hilbert spaces generate an Arveson system of type $
I\!I $ or $ I\!I\!I $? The question is still open for $ n>2 $.
\end{abstract}

\setcounter{tocdepth}{1}
\tableofcontents

\section[Discrete time]
 {\raggedright Discrete time}
\label{sec:1}
Subproduct systems of Hilbert spaces and Hilbert modules are introduced
recently by Shalit and Solel \cite{SS} and Bhat and Mukherjee \cite{BM}.

In this section, by a subproduct system I mean a discrete-time subproduct
system of two-dimensional Hilbert spaces (over $ \C $), defined as follows.

\begin{definition}\label{subproduct_system}
A \emph{subproduct system} consists of two-dimensional Hilbert spaces $ E_t $
for $ t = 1,2,\dots $ and isometric linear maps
\[
\be_{s,t} : E_{s+t} \to E_s \otimes E_t
\]
for $ s,t \in \{1,2,\dots\} $, satisfying the associativity condition:
the diagram\footnote{%
 Of course, $ \One_t : E_t \to E_t $ is the identity map.}
\[
\xymatrix{
 & E_{r+s+t} \ar[dl]_{\be_{r+s,t}} \ar[dr]^{\be_{r,s+t}}
\\
 E_{r+s} \otimes E_t \ar[dr]_{\be_{r,s}\otimes\One_t} & & E_r \otimes E_{s+t}
 \ar[dl]^{\One_r\otimes\be_{s,t}}
\\
 & E_r \otimes E_s \otimes E_t
}
\]
is commutative for all $ r,s,t \in \{1,2,\dots\} $.
\end{definition}

\begin{proposition}\label{1.2}
For every subproduct system there exist vectors $ x_t, y_t \in E_t $ for $ t
= 1,2,\dots $ such that one and only one of the following five conditions is
satisfied.

(1) There exists $ a \in [0,1) $ such that for all $ s,t \in \{1,2,\dots\} $
\begin{gather*}
\| x_t \| = \| y_t \| = 1 \, , \quad \ip{ x_t }{ y_t } = a^t \, , \\
\be_{s,t} (x_{s+t}) = x_s \otimes x_t \, , \quad
\be_{s,t} (y_{s+t}) = y_s \otimes y_t \, .
\end{gather*}

(2) There exists $ a \in [0,1) $ such that for all $ s,t \in \{1,2,\dots\} $
\begin{gather*}
\| x_t \| = \| y_t \| = 1 \, , \quad \ip{ x_t }{ y_t } = a^t \, , \\
\be_{s,t} (x_{s+t}) = \begin{cases}
 x_s \otimes x_t &\text{if $ s $ is even},\\
 x_s \otimes y_t &\text{if $ s $ is odd},
\end{cases} \quad \be_{s,t} (y_{s+t}) = \begin{cases}
 y_s \otimes y_t &\text{if $ s $ is even},\\
 y_s \otimes x_t &\text{if $ s $ is odd}.
\end{cases}
\end{gather*}

(3) There exists $ \la \in \C \setminus \{0\} $ such that for all $ s,t \in
\{1,2,\dots\} $
\begin{gather*}
\| x_t \| = 1 \, , \quad \| y_t \|^2 = 1 + |\la|^2 + |\la|^4 + \dots +
 |\la|^{2t-2} \, , \quad \ip{ x_t }{ y_t } = 0 \, , \\
\be_{s,t} (x_{s+t}) = x_s \otimes x_t \, , \quad
\be_{s,t} (y_{s+t}) = y_s \otimes x_t + \la^s x_s \otimes y_t \, .
\end{gather*}

(4) For all $ s,t \in \{1,2,\dots\} $
\begin{gather*}
\| x_t \| = \| y_t \| = 1 \, , \quad \ip{ x_t }{ y_t } = 0 \, , \\
\be_{s,t} (x_{s+t}) = x_s \otimes x_t \, , \quad
\be_{s,t} (y_{s+t}) = y_s \otimes x_t \, .
\end{gather*}

(5) For all $ s,t \in \{1,2,\dots\} $
\begin{gather*}
\| x_t \| = \| y_t \| = 1 \, , \quad \ip{ x_t }{ y_t } = 0 \, , \\
\be_{s,t} (x_{s+t}) = x_s \otimes x_t \, , \quad
\be_{s,t} (y_{s+t}) = x_s \otimes y_t \, .
\end{gather*}
\end{proposition}

\begin{proof}
The five conditions evidently are pairwise inconsistent; we have to prove that
at least one of them is satisfied.

Ignoring the metric, that is, treating the Hilbert spaces $ E_t $ as just
linear spaces, and using the classification given in \cite[Sect.~7]{I} we get
linearly independent vectors $ x_t, y_t \in E_t $ that satisfy the formulas
for $ \be_{s,t} (x_{s+t}) $ and $ \be_{s,t} (y_{s+t}) $ but maybe violate the
formulas for $ \| x_t \| $, $ \| y_t \| $ and $ \ip{ x_t }{ y_t } $.

We postpone Case 3 and consider Cases 1, 2, 4 and 5. The right-hand sides of
the formulas for $ \be_{s,t} (x_{s+t}) $ and $ \be_{s,t} (y_{s+t}) $ are
tensor products. Taking into account that generally the relation $ x = y
\otimes z \ne 0 $ implies
\[
\frac{ x }{ \| x \| } = \frac{ y }{ \| y \| } \otimes \frac{ z }{ \| z \| } \,
,
\]
we normalize $ x_t, y_t $ (dividing each vector by its norm). Thus,
\[
\| x_t \| = \| y_t \| = 1 \quad \text{for all } t \, .
\]

Case 1. We have $ \ip{ x_{s+t} }{ y_{s+t} } = \ip{ \be_{s,t} (x_{s+t}) }{
\be_{s,t} (y_{s+t}) } = \ip{ x_s \otimes x_t }{ y_s \otimes y_t } = \ip{ x_s
}{ y_s } \ip{ x_t }{ y_t } $, therefore
\[
\ip{ x_t }{ y_t } = a^t \quad \text{for all } t \, ,
\]
where $ a = \ip{ x_1 }{ y_1 } $, $ |a| < 1 $. Replacing $ y_t $ with $ \E^{\I
t\phi} y_t $ for an appropriate $ \phi $ we get $ a \in [0,1) $.

Case 2. Denoting $ \ip{ x_t }{ y_t } $ by $ c_t $ we have
\[
c_{s+t} = \begin{cases}
 c_s c_t &\text{if $ s $ is even},\\
 c_s \overline c_t &\text{if $ s $ is odd},
\end{cases}
\]
thus $ c_{2n} = |c_1|^{2n} $, $ c_{2n+1} = c_1 |c_1|^{2n} $. Replacing $
x_{2n+1} $ with $ \E^{\I \phi} x_{2n+1} $ and $ y_{2n+1} $ with $ \E^{-\I
\phi} y_{2n+1} $  for an appropriate $ \phi $ we get $ c_1 \in [0,1) $.

Case 4. We have $ \ip{ x_{s+t} }{ y_{s+t} } = \ip{ x_s }{ y_s } $, thus $ \ip{
x_t }{ y_t } = \ip{ x_1 }{ y_1 } $. Replacing $ y_t $ with
\[
\frac{ y_t - \ip{ y_t }{ x_t } x_t }{ \| y_t - \ip{ y_t }{ x_t } x_t \| } =
\frac{ y_t - \ip{ y_1 }{ x_1 } x_t }{ \| y_t - \ip{ y_1 }{ x_1 } x_t \| }
\]
we get $ \ip{ x_t }{ y_t } = 0 $.

Case 5 is similar to Case 4.

The rest of the proof treats Case 3.

First, we may prove the condition $ \| y_t \|^2 = 1 + |\la|^2 + \dots +
|\la|^{2t-2} $ for $ t=1 $ only (that is, just $ \| y_1 \| = 1 $), since the
other conditions imply
\[
\| y_{s+t} \|^2 = \| \be_{s,t} (y_{s+t}) \|^2 = \| y_s \otimes x_t \|^2 + \|
\la^s x_s \otimes y_t \|^2 = \| y_s \|^2 + |\la|^{2s} \| y_t \|^2 \, .
\]
Second, the condition $ \| y_1 \| = 1 $ can be ensured by dividing all $ y_t $
by $ \| y_1 \| $. Thus, we need not bother about $ \| y_t \| $.

We still can normalize vectors $ x_t $ (but not $ y_t $), since $ \be_{s,t}
(x_{s+t}) = x_s \otimes x_t $. Thus,
\[
\| x_t \| = 1 \quad \text{for all } t \, .
\]
Denoting $ \ip{ x_t }{ y_t } $ by $ c_t $ we have
\[
c_{s+t} = c_s + \la^s c_t \, .
\]

Sub-case 3a: $ \la = 1 $. We have $ c_{s+t} = c_s + c_t $, thus $ c_t = t c_1 $
for all $ t $. Replacing $ y_t $ with $ y_t - t c_1 x_t $ we get $ \ip{ x_t }{
y_t } = 0 $.

Sub-case 3b: $ \la \ne 1 $. We may replace $ y_t $ with $ y_t + a ( \la^t - 1 )
x_t $ ($ a $ will be chosen later), since
\begin{multline*}
\be_{s,t} ( y_{s+t} + a ( \la^{s+t} - 1 ) x_{s+t} ) = \\
= y_s \otimes x_t + \la^s x_s \otimes y_t + a ( \la^{s+t} - 1 ) x_s \otimes
 x_t = \\
= ( y_s + a ( \la^s - 1 ) x_s ) \otimes x_t + \la^s x_s \otimes ( y_t + a (
\la^t - 1 ) x_t ) \, .
\end{multline*}
An appropriate choice of $ a $ gives $ c_1 = 0 $ and therefore $ c_t = 0 $ for
all $ t $.
\end{proof}

A system of vectors $ x_t, y_t $ cannot satisfy more than one of the five
conditions of Theorem \ref{1.2}, and moreover, the corresponding parameter ($
a $ or $ \la $, if any) is determined uniquely. We say that a system $
(x_t,y_t)_{t=1,2,\dots} $ is a \emph{basis} of type $ \Ec_1(a) $ if it
satisfies Condition (1) of Theorem \ref{1.2} with the given parameter $
a $. Bases of types $ \Ec_2(a) $, $ \Ec_3(\la) $, $ \Ec_4 $ and $ \Ec_5 $ are
defined similarly. By Theorem \ref{1.2}, every subproduct system has a
basis. It can have many bases, but they all are of the same type, which will
be shown in Lemma \ref{1.4}.

\begin{definition}
An \emph{isomorphism} between two subproduct systems $ (E_t,\be_{s,t}) $ and $
(F_t,\ga_{s,t}) $ consists of unitary operators $ \theta_t : E_t \to F_t $ for
$ t=1,2,\dots $ such that the diagram
\[
\xymatrix{
 E_{s+t} \ar[d]_{\theta_{s+t}} \ar[r]^{\be_{s,t}} & E_s \otimes E_t
 \ar[d]^{\theta_s \otimes \theta_t}
\\
 F_{s+t} \ar[r]^{\ga_{s,t}} & F_s \otimes F_t
}
\]
is commutative for all $ t \in \{1,2,\dots\} $.
\end{definition}

Subproduct systems having a basis of type $ \Ec_1(a) $ for a given $ a $
evidently exist and evidently are mutually isomorphic. The same holds for
the other types. Up to isomorphism we have subproduct systems 
\begin{xalignat*}{2}
& \Ec_1 (a) & & \text{for } a \in [0,1) \, , \\
& \Ec_2 (a) & & \text{for } a \in [0,1) \, , \\
& \Ec_3(\la) & & \text{for } \la \in \C \setminus \{0\} \, , \\
& \Ec_4 \, , \\
& \Ec_5 \, .
\end{xalignat*}

\begin{lemma}\label{1.4}
The subproduct systems listed above are mutually non-isomorphic.
\end{lemma}

\begin{proof}
Ignoring the metric we get subproduct systems of linear spaces classified in
\cite[Sect.~7]{I}:
\[
\begin{array}{ccc}
\text{\small subproduct system} & & \text{\small subproduct system of} \\
\text{\small of Hilbert spaces} & & \text{\small linear spaces, according to \cite{I}} \\
\\
\Ec_1(a) & & \Ec_1 \\
\Ec_2(a) & & \Ec_2 \\
\Ec_3(\la) & & \Ec_3(\la) \\
\Ec_4 & & \Ec_4 \\
\Ec_5 & & \Ec_5
\end{array}
\]
Isomorphism between subproduct systems of linear spaces is necessary for
isomorphism between subproduct systems of Hilbert spaces. It remains to prove
that $ \Ec_1(a) $ and $ \Ec_1(b) $ are isomorphic for $ a=b $ only, and the
same for $ \Ec_2(a), \Ec_2(b) $.

Case $ \Ec_1(a) $. All product vectors in $ \be_{1,1} (E_2) $ are of the form
$ c x_1 \otimes x_1 $ or $ c y_1 \otimes y_1 $ ($ c \in \C $), see \cite[Lemma
3.2]{I}. Two one-dimensional subspaces of $ E_1 $ are thus singled out. The
cosine of the angle between them is an invariant, equal to the parameter $ a
$.

Case $ \Ec_2(a) $ is similar.
\end{proof}

\begin{theorem}
Every subproduct system is isomorphic to one and only one of the subproduct
systems $ \Ec_1(a) $, $ \Ec_2(a) $, $ \Ec_3(\la) $, $ \Ec_4 $, $ \Ec_5 $.
\end{theorem}

\begin{proof}
Combine Prop.~\ref{1.2} and Lemma \ref{1.4}.
\end{proof}

An automorphism of a subproduct system is its isomorphism to itself. Given a
basis $ (x_t,y_t)_t $ of a subproduct system, each automorphism $ (\theta_t)_t
$ transforms it into another basis $ (\theta_t x_t, \theta_t y_t)_t $, which
leads to a bijective correspondence between bases and automorphisms (as long
as the initial basis is kept fixed).  All automorphisms are described below in
terms of a given basis $ (x_t,y_t)_t $. All bases are thus described.

Every subproduct system admits automorphisms
\begin{equation}\label{1.6}
\theta_t = \E^{\I ct} \cdot \One_t
\end{equation}
for $ c \in [0,2\pi) $, called trivial automorphisms. They commute with all
automorphisms.

The systems $ \Ec_1(a) $ and $ \Ec_2(a) $ admit a nontrivial automorphism
\begin{equation}\label{1.7}
\theta_t x_t = y_t \, , \quad \theta_t y_t = x_t \, .
\end{equation}
The system $ \Ec_1(0) $ (that is, $ \Ec_1(a) $ for $ a = 0 $) admits
nontrivial automorphisms
\begin{equation}\label{1.8}
\theta_t x_t = x_t \, , \quad \theta_t y_t = \E^{\I bt} y_t
\end{equation}
for $ b \in [0,2\pi) $. The system $ \Ec_2(0) $ admits nontrivial
automorphisms
\begin{equation}\label{1.9}
\begin{gathered}
\theta_{2n} x_{2n} = x_{2n} \, , \quad \theta_{2n} y_{2n} = y_{2n} \, , \\
\theta_{2n-1} x_{2n-1} = \E^{\I b} x_{2n-1} \, , \quad \theta_{2n-1}
y_{2n-1} = \E^{-\I b} y_{2n-1}
\end{gathered}
\end{equation}
for $ b \in [0,2\pi) $.

The systems $ \Ec_3(\la) $, $ \Ec_4 $ and $ \Ec_5 $ admit nontrivial
automorphisms
\begin{equation}\label{1.10}
\theta_t x_t = x_t \, , \quad \theta_t y_t = \E^{\I b} y_t
\end{equation}
for $ b \in [0,2\pi) $.

\begin{lemma}\label{1.11}
For the systems $ \Ec_1(a) $ and $ \Ec_2(a) $ with $ a \ne 0 $, every
automorphism is the composition of a trivial automorphism and possibly the
automorphism \eqref{1.7}.

For the system $ \Ec_1(0) $, every automorphism is the composition of a
trivial automorphism, possibly the automorphism \eqref{1.7}, and possibly an
automorphism of the form \eqref{1.8}.

For the system $ \Ec_2(0) $, every automorphism is the composition of a
trivial automorphism, possibly the automorphism \eqref{1.7}, and possibly an
automorphism of the form \eqref{1.9}.

For the systems $ \Ec_3(\la) $, $ \Ec_4 $ and $ \Ec_5 $, every automorphism is
the composition of a trivial automorphism and possibly an automorphism of the
form \eqref{1.10}.
\end{lemma}

\begin{proof}
By the definition of automorphism, $ \be_{s,t} ( \theta_{s+t} u ) = (
\theta_s \otimes \theta_t ) ( \be_{s,t}(u) ) $ for all $ u \in E_{s+t} $.

For $ \Ec_4 $ we have $ \be_{s,t} (E_{s+t}) = E_s \otimes x_t $, therefore $
\theta_t x_t = \E^{\I c_t} x_t $ for some $ c_t $. The relation $ \be_{s,t}
(x_{s+t}) = x_s \otimes x_t $ implies $ \E^{\I c_{s+t}} = \E^{\I c_s} \E^{\I
c_t} $, since
\begin{multline*}
\E^{\I c_{s+t}} x_s \otimes x_t = \be_{s,t} \( \E^{\I c_{s+t}} x_{s+t} \) = \\
= \be_{s+t} \( \theta_{s+t} x_{s+t} \) = ( \theta_s \otimes \theta_t ) (
 \be_{s,t}(x_{s+t}) ) = \\
= ( \theta_s \otimes \theta_t ) ( x_s \otimes x_t ) = ( \theta_s x_s ) \otimes
( \theta_t x_t ) = \E^{\I c_s} x_s \otimes \E^{\I c_t} x_t \, ;
\end{multline*}
thus, $ \theta_t x_t = \E^{\I ct} x_t $. On the other hand, $ \theta_t y_t =
\E^{\I \al_t } y_t $ for some $ \al_t $, since $ \theta_t y_t $ is a unit
vector orthogonal to $ \theta_t x_t $. The relation $ \be_{s,t} (y_{s+t}) =
y_s \otimes x_t $ implies $ \E^{\I \al_{s+t}} = \E^{\I\al_s} \E^{\I ct} $;
thus $ \theta_t y_t = \E^{\I b} \E^{\I ct} y_t $ where $ b = \al_1 - c $.

For $ \Ec_5 $ the proof is similar.

For $ \Ec_3(\la) $ we note that all product vectors in $ \be_{s,t} (E_{s+t}) $
are collinear to $ x_s \otimes x_t $ (see \cite[Lemma 3.3]{I}), therefore $
\theta_t x_t = \E^{\I c_t} x_t $ for some $ c_t $. As before, the relation $
\be_{s,t} (x_{s+t}) = x_s \otimes x_t $ implies $ \E^{\I c_{s+t}} = \E^{\I
c_s} \E^{\I c_t} $, thus $ \theta_t x_t = \E^{\I ct} x_t $. On the other hand,
$ \theta_t y_t = \E^{\I \al_t } y_t $ for some $ \al_t $, since $ \theta_t y_t
$ is a vector orthogonal to $ \theta_t x_t $. As before, the relation $
\be_{s,t} (y_{s+t}) = y_s \otimes x_t + \la^s x_s \otimes y_t $ implies
\[
\E^{\I \al_{s+t}} ( y_s \otimes x_t + \la^s x_s \otimes y_t ) = \E^{\I \al_s}
y_s \otimes \E^{\I ct} x_t + \la^s \E^{\I cs} x_s \otimes \E^{\I \al_t} y_t
\, ,
\]
therefore $ \E^{\I \al_{s+t}} = \E^{\I \al_s} \E^{\I ct} $ and $ \theta_t y_t
= \E^{\I b} \E^{\I ct} y_t $.

For $ \Ec_1(a) $ we recall the argument used in the proof of Lemma \ref{1.4}:
two one-dimensional subspaces are singled out; $ \theta_t x_t $ must be either
$ \E^{\I c_t} x_t $ or $ \E^{\I c_t} y_t $, and the same holds for $ \theta_t
y_t $. The relation $ \be_{s,t} (x_{s+t}) = x_s \otimes x_t $ implies $ \E^{\I
c_{s+t}} = \E^{\I c_s} \E^{\I c_t} $ as before, but it also shows that the
choice of $ x $ or $ y $ must conform at $ s $, $ t $ and $ s+t $, which means
that either $ \theta_t x_t = \E^{\I ct} x_t $ for all $ t $, or $ \theta_t
x_t = \E^{\I ct} y_t $ for all $ t $. Accordingly, in the first case $
\theta_t y_t = \E^{\I dt} y_t $ for all $ t $, while in the second case $
\theta_t y_t = \E^{\I dt} x_t $ for all $ t $.

Assume for now that $ a \ne 0 $. The relation $ \ip{ x_t }{ y_t } = a^t $
gives $ \E^{\I ct} \E^{-\I dt} = 1 $ and so, $ d = c $. The first case leads
to the trivial automorphism
\[
\theta_t (x_t) = \E^{\I ct} x_t \, , \quad \theta_t (y_t) = \E^{\I ct} y_t \,
.
\]
The second case leads to
\[
\theta_t (x_t) = \E^{\I ct} y_t \, , \quad \theta_t (y_t) = \E^{\I ct} x_t \,
,
\]
the composition of a trivial automorphism and \eqref{1.7}.

Assume now that $ a = 0 $. The first case leads to
\[
\theta_t (x_t) = \E^{\I ct} x_t \, , \quad \theta_t (y_t) = \E^{\I dt} y_t \,
,
\]
the composition of a trivial automorphism and \eqref{1.8}.
The second case leads to
\[
\theta_t (x_t) = \E^{\I ct} y_t \, , \quad \theta_t (y_t) = \E^{\I dt} x_t \,
,
\]
the composition of a trivial automorphism, \eqref{1.7} and \eqref{1.8}.

For $ \Ec_2(a) $: this case is left to the reader. (It will be excluded in
Sect.~\ref{sec:3}, anyway.)
\end{proof}

\section[Time on a sublattice]
 {\raggedright Time on a sublattice}
\label{sec:2}
Let $ (E_t,\be_{s,t})_{s,t} $ be a subproduct system, and $ m \in
\{1,2,\dots\} $. Restricting ourselves to $ E_m, E_{2m}, \dots $ we get
another subproduct system $ (E_{mt},\be_{ms,mt})_{s,t} $. The type of the new
system is uniquely determined by the type of the original system:
\begin{equation}\label{ooo}
\begin{array}{ccc}
\text{type of } (E_t)_t & & \text{type of } (E_{mt})_t \\
\\
\Ec_1(a) & & \Ec_1 (a^m) \\
\Ec_2(a) & & \begin{cases}
 \Ec_1 (a^m)& \text{if $ m $ is even},\\
 \Ec_2 (a^m)& \text{if $ m $ is odd}
\end{cases} \\
\Ec_3(\la) & & \Ec_3(\la^m) \\
\Ec_4 & & \Ec_4 \\
\Ec_5 & & \Ec_5
\end{array}
\end{equation}
Here and later on I abbreviate $ (E_t,\be_{s,t})_{s,t} $ to $ (E_t)_t $. The
proof of \eqref{ooo} is straightforward: every basis $ b = (x_t,y_t)_t $ of $
(E_t)_t $ leads naturally to a basis $ R_m (b) $ of $ (E_{mt})_t $, namely,
\begin{equation}\label{2.2}
R_m : (x_t,y_t)_t \mapsto (x_{mt},y_{mt})_t
\end{equation}
when $ (E_t)_t $ is of type $ \Ec_1(a) $, $ \Ec_2(a) $, $ \Ec_4 $ or $ \Ec_5
$, and
\begin{equation}\label{2.3}
R_m : (x_t,y_t)_t \mapsto \Big( x_{mt}, \frac1{\|y_m\|} y_{mt} \Big)_t
\end{equation}
when $ (E_t)_t $ is of type $ \Ec_3(\la) $.

Every automorphism $ \Theta = (\theta_t)_t $ of $ (E_t)_t $ leads naturally to
an automorphism $ S_m(\Theta) = (\theta_{mt})_t $ of $ (E_{mt})_t $. Clearly,
\[
R_m \( \Theta (b) \) = (S_m(\Theta)) \( R_m(b) \) \, ,
\]
where automorphisms act naturally on bases:
\[
\Theta (b) = ( \theta_t x_t, \theta_t y_t )_t \quad \text{for } \Theta =
(\theta_t)_t \text{ and } b = (x_t,y_t)_t \, .
\]
Note also that
\begin{equation}\label{2.4}
R_{mn} (b) = R_m ( R_n(b) ) \, ,
\end{equation}
and $ S_{mn} (\Theta) = S_m ( S_n(\Theta) ) $ for all $ m,n $, $ b $ and $
\Theta $.

The map $ S_m $ is a homomorphism from the group of automorphisms of $ (E_t)_t
$ to the group of automorphisms of $ (E_{mt})_t $.

\begin{lemma}
For every $ m $ the homomorphism $ S_m $ is an epimorphism. That is, every
automorphism of $ (E_{mt})_t $ is of the form $ S_m(\Theta) $.
\end{lemma}

\begin{proof}
By Lemma \ref{1.11}, every automorphism of $ (E_{mt})_t $ is the product of
automorphisms written out explicitly in \eqref{1.6}--\eqref{1.10}. It is
sufficient to check that each factor is of the form $ S_m(\Theta) $. The
check, left to the reader, is straightforward. (The case $ \Ec_2(a) $ will be
excluded in Sect.~\ref{sec:3}, anyway.)
\end{proof}

\begin{corollary}\label{2.6}
For every $ m $ the map $ R_m $ is surjective. That is, every basis of $
(E_{mt})_t $ is of the form $ R_m(b) $ where $ b $ is a basis of $ (E_t)_t $.
\end{corollary}

\begin{proof}
Let $ b_0 $ be a basis of $ (E_t)_t $ and $ \Theta $ run over all
automorphisms of $ (E_t)_t $, then $ S_m(\Theta) $ runs over all automorphisms
of $ (E_{mt})_t $, therefore $ R_m ( \Theta(b_0) ) = (S_m(\Theta)) ( R_m(b_0)
) $ runs over all bases of $ (E_{mt})_t $.
\end{proof}

\section[Rational time]
 {\raggedright Rational time}
\label{sec:3}
In this section the object introduced by Def.~\ref{subproduct_system} will be
called a \emph{discrete-time subproduct system.} A \emph{rational-time
subproduct system} is defined similarly; in this case the variables $ r,s,t $
of Def.~\ref{subproduct_system} run over the set $ \Q_+ = \Q \cap (0,\infty) $
of all positive rational numbers (rather than the set $ \Z_+ = \Z \cap
(0,\infty) $ of all positive integers).

\begin{theorem}\label{3.1}
For every rational-time subproduct system there exist vectors $ x_t, y_t \in
E_t $ for $ t \in \Q_+ $ such that one and only one of the following four
conditions (numbered 1, 3, 4, 5) is satisfied.

(1) There exists $ a \in [0,1) $ such that for all $ s,t \in \Q_+ $
\begin{gather*}
\| x_t \| = \| y_t \| = 1 \, , \quad \ip{ x_t }{ y_t } = a^t \, , \\
\be_{s,t} (x_{s+t}) = x_s \otimes x_t \, , \quad
\be_{s,t} (y_{s+t}) = y_s \otimes y_t \, .
\end{gather*}

(3) There exist $ c \in (0,\infty) $ and a family $ (\eta_t)_{t\in\Q_+} $ of
numbers $ \eta_t \in \C $ such that for all $ s,t \in \Q_+ $ we have $
|\eta_t| = 1 $, $ \eta_{s+t} = \eta_s \eta_t $ and
\begin{gather*}
\| x_t \| = 1 \, , \quad \ip{ x_t }{ y_t } = 0 \, , \\
\| y_t \|^2 =  \begin{cases}
  t &\text{if } c = 1,\\
  (c^{2t}-1) / (c^2-1) &\text{if } c \ne 1;
 \end{cases} \\
\be_{s,t} (x_{s+t}) = x_s \otimes x_t \, , \quad
\be_{s,t} (y_{s+t}) = y_s \otimes x_t + c^s \eta_s x_s \otimes y_t \, .
\end{gather*}

(4) For all $ s,t \in \Q_+ $
\begin{gather*}
\| x_t \| = \| y_t \| = 1 \, , \quad \ip{ x_t }{ y_t } = 0 \, , \\
\be_{s,t} (x_{s+t}) = x_s \otimes x_t \, , \quad
\be_{s,t} (y_{s+t}) = y_s \otimes x_t \, .
\end{gather*}

(5) For all $ s,t \in \Q_+ $
\begin{gather*}
\| x_t \| = \| y_t \| = 1 \, , \quad \ip{ x_t }{ y_t } = 0 \, , \\
\be_{s,t} (x_{s+t}) = x_s \otimes x_t \, , \quad
\be_{s,t} (y_{s+t}) = x_s \otimes y_t \, .
\end{gather*}
\end{theorem}

\begin{proof}
For each $ n = 1,2,\dots $ we consider the restriction of $ (E_t)_{t\in Q_+} $
to $ t \in \{ 1/n!, 2/n!, \dots \} $, the discrete-time subproduct system $
(E_{t/n!})_{t\in\Z_+} $. The systems $ (E_{t/(n+1)!})_{t\in\Z_+} $ and $
(E_{t/n!})_{t\in\Z_+} $ are related as described in Sect.~2 (with $ m = n+1
$). The first of such systems, $ (E_t)_{t\in\Z_+} $, can be of type $
\Ec_1(a) $, $ \Ec_3(\la) $, $ \Ec_4 $ or $ \Ec_5 $ but not $ \Ec_2(a) $ (since
there is no corresponding type of $ (E_{t/2})_{t\in\Z_+} $). We postpone the
case $ \Ec_3(\la) $ and assume for now that $ (E_t)_{t\in\Z_+} $ is of type $
\Ec_1(a) $, $ \Ec_4 $ or $ \Ec_5 $. Accordingly, the $n$-th of these systems,
$ (E_{t/n!})_{t\in\Z_+} $, is of type $ \Ec_1(a^{1/n!}) $, $ \Ec_4 $ or $
\Ec_5 $. Using Corollary \ref{2.6} we choose for each $ n=1,2,\dots $ a basis
$ b_{1/n!} $ of $ (E_{t/n!})_{t\in\Z_+} $ such that $ R_{n+1} (b_{1/(n+1)!}) =
b_{1/n!} $ for each $ n $.

Every $ r \in \Q_+ $ is of the form $ {k_n}/n! $ for $ n $ large enough. The
basis $ b_r = R_{k_n} (b_{1/n!}) $ of $ (E_{rt})_{t\in\Z_+} $ does not depend on $
n $ due to \eqref{2.4}. For the same reason,
\[
b_{mr} = R_m (b_r) \quad \text{for all } r \in \Q_+ \text{ and } m=1,2,\dots
\]
We introduce $ x^{(r)}_{rt}, y^{(r)}_{rt} \in E_{rt} $ for $ t \in \Z_+ $ by
\[
b_r = \( x^{(r)}_{rt}, y^{(r)}_{rt} \)_{t\in\Z_+} \, .
\]

Case $ \Ec_4 $: by \eqref{2.2},
\[
\( x^{(mr)}_{mrt}, y^{(mr)}_{mrt} \)_{t\in\Z_+} = b_{mr} = R_m (b_r) =
\( x^{(r)}_{mrt}, y^{(r)}_{mrt} \)_{t\in\Z_+} \, ,
\]
that is, $ x^{(mr)}_{mrt} = x^{(r)}_{mrt} = x_{mrt} $ depends only on $ mrt $;
vectors $ x_t \in E_t $ for $ t \in \Q_+ $ are thus defined. The same holds
for $ y_t $. Clearly, $ \| x_t \| = \| y_t \| = 1 $ and $ \ip{ x_t }{ y_t } =
0 $. The relation $ \be_{s,t} (x_{s+t}) = x_s \otimes x_t $ for $ s,t \in \Q_+
$ follows from the relation
\[
\be_{k/n!,l/n!} \( x^{(1/n!)}_{(k+l)/n!} \) = x^{(1/n!)}_{k/n!} \otimes
x^{(1/n!)}_{l/n!} 
\]
for $ k,l \in \Z_+ $ and $ n $ large enough. The same holds for $ y $.

Case $ \Ec_5 $ is similar.

Case $ \Ec_1(a) $: as before, \eqref{2.2} leads to $ x_t,y_t \in E_t $ for $ t
\in \Q_+ $; also $ \| x_t \| = \| y_t \| = 1 $, and $ \be_{s,t} (x_{s+t}) =
x_s \otimes x_t $, $ \be_{s,t} (y_{s+t}) = y_s \otimes y_t $. The relation $
\ip{ x_t }{ y_t } = a^t $ follows from the relation
\[
\ip{ x_{k/n!}^{(1/n!)} }{ y_{k/n!}^{(1/n!)} } = \( a^{1/n!} \)^k
\]
for the system $ (E_{t/n!})_{t\in\Z_+} $ of type $ \Ec_1(a^{1/n!}) $.

Case $ \Ec_3(\la) $: the system $ (E_{t/n!})_{t\in\Z_+} $ is of type $
\Ec_3(\la_{1/n!}) $ for some $ \la_{1/n!} \in \C \setminus \{0\} $ satisfying
$ \la_{1/n!}^{n!} = \la $ and $ \la_{1/(n+1)!}^{n+1} = \la_{1/n!} $. We define
$ \la_r $ for all $ r \in \Q_+ $ by $ \la_{k/n!} = \la_{1/n!}^k $ and get $
\la_{r+s} = \la_r \la_s $ and $ |\la_s| = |\la|^s $. By \eqref{2.3},
\[
\( x^{(mr)}_{mrt}, y^{(mr)}_{mrt} \)_{t\in\Z_+} = b_{mr} = R_m (b_r) =
\Big( x^{(r)}_{mrt}, \frac1{ \| y^{(r)}_{mr} \| } y^{(r)}_{mrt}
\Big)_{t\in\Z_+} \, ;
\]
taking $ r = \frac1{(n+1)!} $ and $ m = n+1 $ we get for $ t \in \Z_+ $
\[
x^{(1/n!)}_{t/n!} = x^{(1/(n+1)!)}_{t/n!} \, , \quad
y^{(1/n!)}_{t/n!} = \frac1{ \| y^{(1/(n+1)!)}_{1/n!} \| }
y^{(1/(n+1)!)}_{t/n!} \, .
\]
We define $ x_r $ for all $ r \in \Q_+ $ by $ x_{t/n!} = x^{(1/n!)}_{t/n!} $
and get $ \be_{s,t} (x_{s+t}) = x_s \otimes x_t $ as before; however, $ y $
needs more effort. We have
\[
\| y_{1/n!}^{(1/(n+1)!)} \|^2 = \| y_{(n+1)/(n+1)!}^{(1/(n+1)!)} \|^2 = 1 +
c^{2/(n+1)!} + \dots + c^{2n/(n+1)!} \, ,
\]
where $ c = |\la| $.

Sub-case $ c = 1 $: we note that $ \| y_{1/n!}^{(1/(n+1)!)} \|^2 = n+1 $
implies
\[
\frac1{ \sqrt{n!} } y^{(1/n!)}_{t/n!} = \frac1{ \sqrt{(n+1)!} }
y^{(1/(n+1)!)}_{(n+1)t/(n+1)!} \, ,
\]
define $ y_r $ for all $ r \in \Q_+ $ by
\[
y_{t/n!} = \frac1{ \sqrt{n!} } y^{(1/n!)}_{t/n!}
\]
and get for $ t \in \Z_+ $
\[
\| y_{t/n!} \|^2 = \frac1{n!} \| y^{(1/n!)}_{t/n!} \|^2 = \frac{t}{n!} \, ,
\]
that is, $ \| y_t \|^2 = t $ for all $ t \in \Q_+ $. The relation
\[
\be_{k/n!,l/n!} \( y_{(k+l)/n!}^{(1/n!)} \) = y_{k/n!}^{(1/n!)} \otimes
x_{l/n!}^{(1/n!)} + \la_{1/n!}^k x_{k/n!}^{(1/n!)} \otimes y_{l/n!}^{(1/n!)}
\]
implies $ \be_{s,t} (y_{s+t}) = y_s \otimes x_t + \la_s x_s \otimes y_t $. It
remains to define $ \eta_s $ by
\[
\la_s = |\la_s| \eta_s = c^s \eta_s \, .
\]

Sub-case $ c \ne 1 $: only the calculations related to $ \| y_t \| $ must be
reconsidered. We note that
\[
\| y_{1/n!}^{(1/(n+1)!)} \|^2 = \frac{ c^{2/n!} - 1 }{ c^{2/(n+1)!} - 1 }
\]
implies
\[
\sqrt{ c^{2/n!}-1 } y^{(1/n!)}_{t/n!} = \sqrt{ c^{2/(n+1)!}-1 }
y^{(1/(n+1)!)}_{(n+1)t/(n+1)!} \, ,
\]
define $ y_r $ for all $ r \in \Q_+ $ by
\[
y_{t/n!} = \sqrt{ \frac{ c^{2/n!}-1 }{ c^2-1 } } y^{(1/n!)}_{t/n!} \quad
\text{for } t \in \Z_+
\]
and get
\[
\| y_{t/n!} \|^2 = \frac{ c^{2/n!}-1 }{ c^2-1 } \cdot \frac{ c^{2t/n!}-1 }{
c^{2/n!}-1 } = \frac{ c^{2t/n!}-1 }{ c^2-1 } \quad \text{for } t \in \Z_+ \, ,
\]
that is, $ \| y_t \|^2 = (c^{2t}-1)/(c^2-1) $ for all $ t \in \Q_+ $.
\end{proof}

\section[Arveson systems, Liebscher continuity]
 {\raggedright Arveson systems, Liebscher continuity}
\label{sec:4}
An Arveson system, or product system of Hilbert spaces, consists of separable
Hilbert spaces $ \ti E_t $ for $ t \in \R_+ = (0,\infty) $ and unitary
operators
\[
\ti\be_{s,t} : \ti E_{s+t} \to \ti E_s \otimes \ti E_t
\]
for $ s,t \in \R_+ $, satisfying well-known conditions, see
\cite[Def.~3.1.1]{Ar}.

\begin{definition}\label{4.1}
Let $ (E_t,\be_{s,t})_{s,t\in\Q_+} $ be a rational-time subproduct system (as
defined in Sect.~\ref{sec:3}) and $ (\ti E_t,\ti\be_{s,t})_{s,t\in\R_+} $ an
Arveson system. A \emph{representation} of the subproduct system in the
Arveson system consists of linear isometric embeddings
\[
\al_t : E_t \to \ti E_t \quad \text{for } t \in \Q_+
\]
such that the diagram
\[
\xymatrix{
 E_{s+t} \ar[d]_{\al_{s+t}} \ar[r]^{\be_{s,t}} & E_s \otimes E_t
 \ar[d]^{\al_s \otimes \al_t}
\\
 \ti E_{s+t} \ar[r]^{\ti\be_{s,t}} & \ti E_s \otimes \ti E_t
}
\]
is commutative for all $ s,t \in \Q_+ $.
\end{definition}

Denote by $ \xi_{s,t} : E_s \otimes E_t \to E_t \otimes E_s $ the unitary
exchange operator, $ \xi_{s,t} ( f \otimes g ) = g \otimes f $.

\begin{proposition}\label{4.2}
If $ (E_t,\be_{s,t})_{s,t\in\Q_+} $ admits a representation (in some Arveson
system) then for every $ h \in E_1 $ the function
\[
t \mapsto \begin{cases}
 \ip{ \be_{t,1-t}(h) }{ \xi_{1-t,t} \be_{1-t,t} (h) } &\text{for } t \in \Q
  \cap (0,1), \\
 \| h \|^2 &\text{for } t \in \{0,1\}
\end{cases}
\]
is uniformly continuous on $ \Q \cap [0,1] $.
\end{proposition}

\begin{proof}
By a theorem of Liebscher \cite[Th.~7.7]{Li}, for every Arveson system $ (\ti
E_t,\ti\be_{s,t})_{s,t\in\R_+} $ the unitary operators $ U_t : \ti E_1 \to \ti
E_1 $ defined for $ t \in (0,1) $ by
\[
\ti\be_{t,1-t} ( U_t \ti h ) = \ti\xi_{1-t,t} \ti\be_{1-t,t} (\ti h) \quad
\text{for } \ti h \in \ti E_1
\]
are a strongly continuous one-parameter unitary group (or rather its
restriction to the time interval $ (0,1) $) such that $ U_1 = U_0 = \One $; of
course, $ \ti\xi_{s,t} : \ti E_s \otimes \ti E_t \to \ti E_t \otimes \ti E_s $
is the exchange operator. It follows that for every $ \ti h \in \ti E_1 $ the
function $ t \mapsto \ip{ U_t \ti h }{ \ti h } $ is (uniformly) continuous on
$ [0,1] $. We may rewrite it as
\[
\ip{ U_t \ti h }{ \ti h } = \ip{ \ti\be_{t,1-t} (U_t \ti h) }{ \ti\be_{t,1-t}
(\ti h) } = \ip{ \ti\xi_{1-t,t} \ti\be_{1-t,t} (\ti h) }{ \ti\be_{t,1-t}(\ti
h) } \, .
\]

Given a representation $ (\al_t)_t $ and a vector $ h \in E_1 $, we take $ \ti
h = \al_1(h) \in \ti E_1 $ and get
\begin{gather*}
\ti\be_{t,1-t}(\ti h) = ( \al_t \otimes \al_{1-t} ) ( \be_{t,1-t}(h) ) \, ; \\
\ti\xi_{1-t,t} \ti\be_{1-t,t} (\ti h) = ( \al_t \otimes \al_{1-t} )
 \xi_{1-t,t} \be_{1-t,t} (h) \, ;
\end{gather*}
thus,
\[
\ip{ U_t \ti h }{ \ti h } = \ip{ \xi_{1-t,t} \be_{1-t,t} (h) }{ \be_{t,1-t}(h)
}
\]
for all $ t \in \Q \cap (0,1) $. Also, $ \ip{ U_1 \ti h }{ \ti h } = \ip{ U_0
\ti h }{ \ti h } = \| \ti h \|^2 = \| h \|^2 $.
\end{proof}

\begin{corollary}\label{4.3}
A rational-time subproduct system satisfying Condition (4) of Theorem \ref{3.1}
admits no representation. The same holds for Condition (5).
\end{corollary}

\begin{proof}
We take $ h = y_1 $ and get $ \| h \|^2 = 1 $ but
\[
\ip{ \be_{t,1-t}(h) }{ \xi_{1-t,t} \be_{1-t,t} (h) } = \ip{ y_t \otimes
x_{1-t} }{ x_t \otimes y_{1-t} } = 0
\]
for all $ t \in \Q \cap (0,1) $. Condition (5) is treated similarly.
\end{proof}

\begin{lemma}\label{4.4}
If a rational-time subproduct system satisfying Condition (3) of Theorem
\ref{3.1} admits a representation then there exists $ b \in \R $ such that $
\eta_t = \E^{\I bt} $ for all $ t \in \Q_+ $.
\end{lemma}

\begin{proof}
It is sufficient to prove that $ \eta_t \to 1 $ as $ t \to 0 $, $ t \in \Q_+
$. We apply Prop.~\ref{4.2} to $ h = y_1 $ and get $ \| h \|^2 = 1 $ but
\begin{multline*}
\ip{ \be_{t,1-t}(h) }{ \xi_{1-t,t} \be_{1-t,t} (h) } = \\
= \ip{ y_t \otimes x_{1-t} + c^t \eta_t x_t \otimes y_{1-t} }{ x_t \otimes
 y_{1-t} + c^{1-t} \eta_{1-t} y_t \otimes x_{1-t} } = \\
= c^t \eta_t \| x_t \otimes y_{1-t} \|^2 + \overline{ c^{1-t} \eta_{1-t} } \|
 y_t \otimes x_{1-t} \|^2 = \\
= \eta_t ( c^t \| y_{1-t} \|^2 + c^{1-t} \overline\eta_1 \| y_t \|^2 ) \, ,
\end{multline*}
since $ \eta_t \eta_{1-t} = \eta_1 $ and $ \overline{ \eta_{1-t} } = 1 /
\eta_{1-t} $. It remains to note that $ c^t \| y_{1-t} \|^2 + c^{1-t}
\overline\eta_1 \| y_t \|^2 \to 1 $ as $ t \to 0 $, $ t \in \Q_+ $, since $ \|
y_t \|^2 \to 0 $.
\end{proof}

\begin{lemma}\label{4.45}
If a rational-time subproduct system satisfying Condition (1) of Theorem
\ref{3.1} admits a representation then $ a \ne 0 $.
\end{lemma}

\begin{proof}
We define $ \ti h \in \ti E_1 $ by
\[
\ti\be_{0.5,0.5} (\ti h) = \al_{0.5} (x_{0.5}) \otimes \al_{0.5} (y_{0.5})
\]
and observe that $ \| \ti h \| = 1 $ but $ \ip{ U_t \ti h }{ \ti h } = a^{2t}
$ for $ t \in \Q \cap (0,0.5) $.
\end{proof}

\begin{theorem}\label{4.5}
For every rational-time subproduct system admitting a representation (in some
Arveson system) there exist vectors $ x_t, y_t \in E_t $ for $ t \in \Q_+ $
such that one and only one of the following two conditions (numbered 1, 3) is
satisfied.

(1) There exists $ a \in (0,1) $ such that for all $ s,t \in \Q_+ $
\begin{gather*}
\| x_t \| = \| y_t \| = 1 \, , \quad \ip{ x_t }{ y_t } = a^t \, , \\
\be_{s,t} (x_{s+t}) = x_s \otimes x_t \, , \quad
\be_{s,t} (y_{s+t}) = y_s \otimes y_t \, .
\end{gather*}

(3) There exist $ c \in (0,\infty) $ and $ b \in \R $ such that for all $ s,t
\in \Q_+ $
\begin{gather*}
\| x_t \| = 1 \, , \quad \ip{ x_t }{ y_t } = 0 \, , \\
\| y_t \|^2 =  \begin{cases}
  t &\text{if } c = 1,\\
  (c^{2t}-1) / (c^2-1) &\text{if } c \ne 1;
 \end{cases} \\
\be_{s,t} (x_{s+t}) = x_s \otimes x_t \, , \quad
\be_{s,t} (y_{s+t}) = y_s \otimes x_t + c^s \E^{\I bs} x_s \otimes y_t \, .
\end{gather*}
\end{theorem}

\begin{proof}
Combine Theorem \ref{3.1}, Corollary \ref{4.3}, Lemma \ref{4.4} and Lemma
\ref{4.45}.
\end{proof}

\begin{lemma}\label{4.6}
Every rational-time subproduct system satisfying Condition (1) of Theorem
\ref{4.5} admits a representation in an Arveson system of type $ I_1 $.
\end{lemma}

\begin{proof}
It is straightforward: in an Arveson system of type $ I_1 $ we choose two
units $ (u_t)_t $, $ (v_t)_t $ such that $ \| u_t \| = 1 $, $ \| v_t \| = 1 $
and $ \ip{ u_t }{ v_t } = a^t $ for all $ t \in [0,\infty) $, and let $
\al_t(x_t) = u_t $, $ \al_t(y_t) = v_t $ for all $ t \in \Q_+ $.
\end{proof}

\begin{lemma}\label{4.7}
Every rational-time subproduct system satisfying Condition (3) of Theorem
\ref{4.5} admits a representation in an Arveson system of type $ I_1 $.
\end{lemma}

\begin{proof}
We take the type $ I_1 $ Arveson system of symmetric Fock spaces
\[
\ti E_t = \bigoplus_{n=0}^\infty \( L_2(0,t) \)^{\otimes n} = \C \oplus
L_2(0,t) \oplus L_2(0,t) \otimes L_2(0,t) \oplus \dots
\]
and map $ E_t $ into the sum of the first two terms,
\begin{gather*}
\al_t : E_t \to \C \oplus L_2(0,t) \subset \ti E_t \, , \\
\al_t (x_t) = 1 \oplus 0 \, , \quad \text{(the vacuum vector)} \\
\al_t (y_t) = 0 \oplus f_t \, , \quad f_t \in L_2(0,t) \text{ is defined by} \\
f_t (s) = A c^s \E^{\I bs} \quad \text{for } s \in (0,t) \, ,
\end{gather*}
$ A $ being the normalizing constant,
\[
A = \begin{cases}
 1 &\text{if } c=1, \\
 \sqrt{ \frac{ 2\ln c }{ c^2-1 } } &\text{if } c \ne 1.
\end{cases}
\]
\end{proof}

A representation (as defined by \ref{4.1}) will be called \emph{reducible,} if
it is also a representation in a proper Arveson subsystem of the given Arveson
system. Otherwise the representation is \emph{irreducible.} All irreducible
representations (if any) of a given rational-time subproduct system are
mutually isomorphic. In this sense a rational-time subproduct system either
extends uniquely to the corresponding Arveson system, or is not embeddable
into Arveson systems.

The representations constructed in Lemmas \ref{4.6}, \ref{4.7} are evidently
irreducible.

\begin{theorem}
If a rational-time subproduct system has an irreducible representation in an
Arveson system then the Arveson system is of type $ I_1 $.
\end{theorem}

\begin{proof}
Combine Theorem \ref{4.5} and Lemmas \ref{4.6}, \ref{4.7}.
\end{proof}

\bigskip
\filbreak
{
\small
\begin{sc}
\parindent=0pt\baselineskip=12pt
\parbox{4in}{
Boris Tsirelson\\
School of Mathematics\\
Tel Aviv University\\
Tel Aviv 69978, Israel
\smallskip
\par\quad\href{mailto:tsirel@post.tau.ac.il}{\tt
 mailto:tsirel@post.tau.ac.il}
\par\quad\href{http://www.tau.ac.il/~tsirel/}{\tt
 http://www.tau.ac.il/\textasciitilde tsirel/}
}

\end{sc}
}
\filbreak

\end{document}